\documentclass[12pt,reqno]{amsart}
\usepackage{amscd,amsmath,amsthm,amssymb}
\usepackage{color}
\usepackage{pstricks}
\usepackage{stmaryrd}
\usepackage{tikz}
\usepackage{url}
\usepackage{lipsum}

\newpsstyle{fatline}{linewidth=1.5pt}
\newpsstyle{fyp}{fillstyle=solid,fillcolor=verylight}
\definecolor{verylight}{gray}{0.97}
\definecolor{light}{gray}{0.9}
\definecolor{medium}{gray}{0.85}
\definecolor{dark}{gray}{0.6}

 \def\NZQ{\mathbb}               

 \def\ZZ{{\NZQ Z}}

 \def\FF{{\NZQ F}}

 \def\frk{\mathfrak}               

 \def\mm{{\frk m}}

 \def\G{{\mathcal G}}

\def\fb{{\mathbf f}}
 \def\opn#1#2{\def#1{\operatorname{#2}}} 
 \opn\chara{char} \opn\length{\ell} \opn\pd{pd} \opn\rk{rk}
 \opn\projdim{proj\,dim} \opn\injdim{inj\,dim} \opn\rank{rank}
 \opn\depth{depth} \opn\grade{grade} \opn\height{height}
 \opn\embdim{emb\,dim} \opn\codim{codim}
 
 \opn\Tr{Tr} \opn\bigrank{big\,rank}
 \opn\superheight{superheight}\opn\lcm{lcm}
 \opn\trdeg{tr\,deg}
 \opn\reg{reg} \opn\lreg{lreg} \opn\ini{in} \opn\lpd{lpd}
 \opn\size{size} \opn\sdepth{sdepth}
 \opn\link{link}\opn\fdepth{fdepth}\opn\lex{lex}
 \opn\tr{tr}
 \opn\type{type}
 \opn\gap{gap}
 \opn\diam{diam}
 \opn\Mod{Mod}
 \opn\div{div} \opn\Div{Div} \opn\cl{cl} \opn\Cl{Cl}
 \opn\Spec{Spec} \opn\Supp{Supp} \opn\supp{supp} \opn\Sing{Sing}
 \opn\Ass{Ass} \opn\Min{Min}\opn\Mon{Mon}
 \opn\Ann{Ann} \opn\Rad{Rad} \opn\Soc{Soc}
 \opn\Im{Im} \opn\Ker{Ker} \opn\Coker{Coker} \opn\Am{Am}
 \opn\Hom{Hom} \opn\Tor{Tor} \opn\Ext{Ext} \opn\End{End}
 \opn\Aut{Aut} \opn\id{id}
 
 \opn\nat{nat}
 \opn\pff{pf}
 \opn\Pf{Pf} \opn\GL{GL} \opn\SL{SL} \opn\mod{mod} \opn\ord{ord}
 \opn\Gin{Gin} \opn\Hilb{Hilb}\opn\sort{sort}
 \opn\PF{PF}\opn\Ap{Ap}
 \opn\dist{dist}
 \opn\aff{aff}
 \opn\relint{relint} \opn\st{st}
 \opn\lk{lk} \opn\cn{cn} \opn\core{core} \opn\vol{vol}  \opn\inp{inp} \opn\nilpot{nilpot}
 \opn\link{link} \opn\star{star}\opn\lex{lex}\opn\set{set}
 \opn\width{wd}
 \opn\Fr{F}
 \opn\QF{QF}
 \opn\G{G}
 \opn\type{type}\opn\res{res}
 \opn\conv{conv}
 \opn\sr{sr}
 \opn\gr{gr}
 
 \def\pot#1#2{#1[\kern-0.28ex[#2]\kern-0.28ex]}

 \opn\dirlim{\underrightarrow{\lim}}
 \opn\inivlim{\underleftarrow{\lim}}
 \def\Implies{\ifmmode\Longrightarrow \else
         \unskip${}\Longrightarrow{}$\ignorespaces\fi}
 \def\implies{\ifmmode\Rightarrow \else
         \unskip${}\Rightarrow{}$\ignorespaces\fi}
 \def\iff{\ifmmode\Longleftrightarrow \else
         \unskip${}\Longleftrightarrow{}$\ignorespaces\fi}

 \let\:=\colon
 \newtheorem{Theorem}{Theorem}[section]
 \newtheorem{Lemma}[Theorem]{Lemma}
 
 \newtheorem{Proposition}[Theorem]{Proposition}

 \newtheorem{Conjecture}[Theorem]{Conjecture}

 \let\epsilon\varepsilon
 \let\kappa=\varkappa
 
 \def\qed{\ifhmode\textqed\fi
       \ifmmode\ifinner\quad\qedsymbol\else\dispqed\fi\fi}
 \def\textqed{\unskip\nobreak\penalty50
        \hskip2em\hbox{}\nobreak\hfill\qedsymbol
        \parfillskip=0pt \finalhyphendemerits=0}
 \def\dispqed{\rlap{\qquad\qedsymbol}}

\begin{document}

\title{The canonical trace of\\ Cohen-Macaulay algebras of codimension 2}
\author{Antonino Ficarra}

\dedicatory{Dedicated with my deepest gratitude to the memory of my maestro and friend,\\ Professor J\"urgen Herzog (1941-2024)}

\address{Antonino Ficarra, Departamento de Matem\'{a}tica, Escola de Ci\^{e}ncias e Tecnologia, Centro de Investiga\c{c}\~{a}o, Matem\'{a}tica e Aplica\c{c}\~{o}es, Instituto de Investiga\c{c}\~{a}o e Forma\c{c}\~{a}o Avan\c{c}ada, Universidade de \'{E}vora, Rua Rom\~{a}o Ramalho, 59, P--7000--671 \'{E}vora, Portugal}
\email{antonino.ficarra@uevora.pt}\email{antficarra@unime.it}


\maketitle
\begin{abstract}
In the present paper, we investigate a conjecture of J\"urgen Herzog. Let $S$ be a local regular ring with residue field $K$ or a positively graded $K$-algebra, $I\subset S$ be a perfect ideal of grade two, and let $R=S/I$ with canonical module $\omega_R$. Herzog conjectured that the canonical trace $\tr(\omega_R)$ is obtained by specialization from the generic case of maximal minors. We prove this conjecture in several cases, and present a criterion that guarantees that the canonical trace specializes under some additional assumptions. As the final conclusion of all of our results, we classify the nearly Gorenstein monomial ideals of height two.
\end{abstract}

\section{Introduction}

Let $(R,\mm,K)$ be either a local ring or a positively graded $K$-algebra. Let $\mm$ be the (graded) maximal ideal of $R$. Assume that $R$ is Cohen-Macaulay and admits a canonical module $\omega_R$. The \textit{canonical trace} of $R$ is then defined as
$$
\tr(\omega_R)\ =\ \sum_{\varphi\in\Hom_R(\omega_R,R)}\varphi(\omega_R).
$$
It turns out that $R$ is a Gorenstein ring if and only if $\tr(\omega_R)=R$. Thus, the canonical trace measures, in a certain sense, how far is $R$ from being Gorenstein.

The ring $R$ is called \textit{nearly Gorenstein} if $\tr(\omega_R)$ contains $\mm$. Nearly Gorenstein rings were introduced Herzog, Hibi and Stamate in \cite{HHS1}, although they were implicitly considered in the work of Huneke and Vraciu \cite{HV}. Many people are interested to develop their theory, classify nearly Gorenstein rings, and compare the nearly Gorenstein property with other properties of $R$. See \cite{D,ET,FHSV22,GHHM,HKMM,HHS2,M24,M24b}.

Let ${\bf x}=x_1,\dots,x_n$ be a regular sequence on $R$, and denote by ``$\overline{\phantom{l.}}$" the reduction modulo ${\bf x}$. It is well-known that the canonical module specializes, that is $\overline{\omega_R}=\omega_{\overline{R}}$. One would guess that the canonical trace specializes as well: $\tr(\omega_{\overline{R}})=\tr(\omega_R)\overline{R}$? It is easy to see that $\tr(\omega_R)\overline{R}\subseteq\tr(\omega_{\overline{R}})$. However, in \cite[Example 3.3]{FHSV22} we discovered a counterexample to this expectation. Let $R=S/I$ where $S=K[x_1,x_2,x_3,y_1,y_2,y_3]$, $I=(x_1y_1,x_2y_2,x_3y_3,x_1x_2,x_1x_3,x_2y_3)$ and let ${\bf x}=x_1-y_1$. Then $\tr(\omega_{\overline{R}})\ne\tr(\omega_R)\overline{R}$. This example is remarkable, because $I$ is the edge ideal of a Cohen-Macaulay very well-covered graph, see \cite{CF1,CRT2011}.

Now, let $S$ be a Noetherian local ring or a positively graded $K$-algebra, and let $I\subset S$ be an ideal, which we assume to be homogeneous if $S$ is graded. We furthermore assume that $S$ is regular. We will explain this assumption in a moment.

We set $R=S/I$, and as before we assume that $R$ is Cohen-Macaulay with canonical module $\omega_R$. The simplest Cohen-Macaulay algebras of this form are those such that $I$ is a perfect ideal of grade two. Recall that $I$ is \textit{perfect} if $\textup{grade}\,R=\pd R$. These algebras are characterized by the celebrated \textit{Hilbert-Burch theorem} \cite{Burch68}.
\begin{Theorem}\label{Thm:HB}
	\textup{(\cite[Theorem 1.4.17]{BH}, \cite[Theorem 20.15]{Ei})} With the notation introduced, assume that $I$ has a free resolution
	$$
	\FF\ \ :\ \ 0\rightarrow S^{m-1}\xrightarrow{\varphi}S^m\rightarrow I\rightarrow0.
	$$
	
	Then, there exists a $S$-regular element $a$ such that $I=aI_{m-1}(\varphi)$. If $I$ is projective, then $I=(a)$, otherwise $\pd I=1$ and $I$ is perfect of grade two.
	
	Conversely, if $a$ is a non-zero divisor on $S$, $\grade I_{m-1}(\varphi)\ge2$ and the map $S^m\rightarrow I$ sends the $i$th basis element to $(-1)^ia$ times the minor obtained from $\varphi$ by leaving out the $i$th row, then $\FF$ is the minimal free resolution of $I=aI_{m-1}(\varphi)$.
\end{Theorem}

Fixing bases of $S^{m-1}$ and $S^m$, we may represent $\varphi$ by a matrix $A$. Any such $A$ is called an \textit{Hilbert-Burch matrix} of $I$. Furthermore, we may compute $I_j(\varphi)$ as $I_j(A)$ for all $j$ and this ideal is independent from the particular Hilbert-Burch matrix $A$. Tensoring $\FF$ with $R=S/I$, we obtain the exact sequence
$$
0\rightarrow C\xrightarrow{\alpha}R^{m-1}\xrightarrow{\psi}R^m\rightarrow I/I^2\rightarrow0,
$$
where $C=\Ker\psi$ is the kernel of $\psi=\varphi\otimes\textup{id}_R$, and $\alpha$ is the inclusion map.

Since, as we assumed before, $S$ is regular, we can use the so-called \textit{Vasconcelos trick} \cite[Corollary 3.2]{HHS1} to obtain that $\tr(\omega_R)=I_1(\alpha)$.

Let now $S=K[x_1,\dots,x_n]$ be a graded polynomial ring over $K$. Let $\mu(I)$ be the minimal number of generators of $I$. The Hilbert-Burch Theorem \ref{Thm:HB} implies that any graded perfect ideal $I$ of grade two, with $\mu(I)=m$, arises as a specialization of the ideal of maximal minors $I_{m-1}(X)$ of a generic matrix $X=(x_{ij})$ of indeterminates. Indeed, if $A=(f_{ij})$ is an Hilbert-Burch matrix of $I$, and $R'=S[X]/I_{m-1}(X)$, then $R\cong R'/({\bf g})$ where ${\bf g}=x_{11}-f_{11}, \ldots, x_{m,m-1}-f_{m,m-1}$ is a regular sequence on $R'$. By \cite[Theorem 1.1]{FHSV22}, we have $\tr(\omega_{K[X]/I_{m-1}(X)})=I_{m-2}(X)/I_{m-1}(X)$. In this setting, we prove in Lemma \ref{Lem:c_A} that
$$
I_{m-2}(\psi)=I_{m-2}(A)R\ \subseteq\ \tr(\omega_{R})=I_1(\alpha).
$$
These facts, and several experimental evidences, lead to the following conjecture which J\"urgen Herzog posed to the author of this paper in private conversations.
\begin{Conjecture}\label{Conj:Her}
	\textup{(Herzog)} Let $S$ be either a regular local ring with residue class field $K$ or a regular positively graded $K$-algebra. Let $I$ be a perfect ideal of grade two and let $A$ be an Hilbert-Burch matrix of $I$. Set $R=S/I$. Then,
	$$
	\tr(\omega_R)=I_{\mu(I)-2}(A)R.
	$$
\end{Conjecture}

In this paper we address this conjecture, and we settle it in important cases. Along the way we characterize perfect monomial ideals of grade two which are generically Gorenstein. More generally, we show that Conjecture \ref{Conj:Her} holds if we assume in addition that $I$ is generically Gorenstein. Furthermore, we prove that Conjecture \ref{Conj:Her} holds for any Cohen-Macaulay monomial ideal $I$ of $S=K[x,y]$, and again we classify the nearly Gorenstein algebras $S/I$ in this class. Finally, we classify the nearly Gorenstein monomial ideals of height two of $S=K[x_1,\dots,x_n]$.

Before breaking out the content of the paper, let us stress the relevance of the Hilbert-Burch Theorem \ref{Thm:HB} across the fifty-five years of J\"urgen Herzog research. Herzog first paper ever \cite{He}, titled \textit{Generators and relations of abelian semigroups and semigroup rings}, contains among other fundamental and basic results on semigroup rings, an explicit description of the minimal free resolution of the semigroup ring of any numerical semigroup with three generators. As Herzog remarked many times in his lifetime, this is just a special case of the Hilbert-Burch theorem, which, however, he did not know at the time as a student.

The Hilbert-Burch theorem gained again Herzog interest in the last few months of his life, as a motivation for Conjecture \ref{Conj:Her}, but also in the study of the so-called \textit{Fitting ideals} $\textup{Fitt}_j(I)$ of an ideal $I$. This led to a joint paper \cite{EFHM} with Eisenbud, Herzog, Moradi and the author of this paper. Such article, the third to last paper of Herzog, deals with the problem of characterizing those ideals $I$ of a Noetherian ring such that $I=\textup{Fitt}_j(I)$ for some $j$. In the case $j=1$, a certain converse of the Hilbert-Burch Theorem \ref{Thm:HB} was proved.

The outlines of the paper are as follows. In Section \ref{sec2} we present a wide class of perfect ideals of grade two for which Conjecture \ref{Conj:Her} holds true. Let $S=K[x_1,\dots,x_n]$ be a positively graded polynomial ring, and let $I\subset S$ be a perfect graded ideal of grade two. It is proved in Theorem \ref{Thm:genericallyci} that, if in addition $I$ is generically Gorenstein, then the canonical trace is the expected one, namely it is generated by the submaximal minors $I_{\mu(I)-2}(X)$, modulo $I$, of any Hilbert-Burch matrix $X$ of $I$. This result was already proved in \cite[Corollary 3.5]{FHSV22} by a different method. Our new proof is based on Lemma \ref{Lem:genGonCI} and some results of Avramov and Herzog \cite{AH}.

On the other hand, the inclusion $I_{\mu(I)-2}(X)R\subseteq\tr(\omega_{R})$ holds true always as we show in Lemma \ref{Lem:c_A}. Next, we characterize the perfect monomial ideals of height two which are generically Gorenstein. It turns out that these ideals are parameterized by finite simple graphs $G$ with $t$ edges and integer sequences ${\bf a}$ and ${\bf b}$ of length $t$ such that a certain graph $G({\bf a,b})$, constructed from such data, is cochordal. This is the content of Theorem \ref{Thm:G(a,b)}. The implication (b) $\Rightarrow$ (c) of the proof of this theorem is in part implicitly contained in the proof of \cite[Proposition 2.2]{Ahmad}.

In Section \ref{sec3} we consider, among monomial ideals, the simplest class of perfect monomial ideals of grade two, which are not necessarily generically Gorenstein. Namely, the Cohen-Macaulay monomial ideals in $S=K[x,y]$. For this class of ideals, we obtain in Lemma \ref{Lem:X-HilBur} a fairly explicit Hilbert-Burch matrix. As a consequence, we are able to show in Theorem \ref{Thm:mainzero} that Conjecture \ref{Conj:Her} holds for this class of ideals. Hence, in Proposition \ref{Prop:I_{a,b}-NG} we successfully classify the nearly Gorenstein ideals within this class.

As the final conclusion of all of our results, in Theorem \ref{Thm:I(G(a,b))-NG} we classify the nearly Gorenstein monomial ideals of height two.

\section{Perfect generically Gorenstein ideals of height two}\label{sec2}

Let $S=K[x_1,\ldots,x_n]$ be a graded polynomial ring over a field $K$ with $\deg x_i=d_i$ for $1\le i\le n$. Since $S$ is Cohen-Macaulay, by \cite[Corollary 2.1.4]{BH} for any ideal $I\subset S$ we have $\grade I=\textup{height}(I)$. Thus $I$ is perfect if and only if $I$ is Cohen-Macaulay.

Recall that $I\subset S$ is called \textit{generically Gorenstein} if $S_P/I_P$ is Gorenstein for any minimal prime $P\in\Ass(I)$. Whereas, $I$ is called \textit{generically a complete intersection} if $S_P/I_P$ is a complete intersection for any minimal prime $P\in\Ass(I)$. Recall that given an ideal $J$ of a Noetherian local (or graded) ring $T$ we have $\mu(J)\ge\height(J)$, and if equality holds $J$ is called a \textit{complete intersection}. If $T$ is Cohen-Macaulay, then $J$ is a complete intersection if and only if $J$ is generated by a regular sequence.

The goal of this section is to prove
\begin{Theorem}
	\label{Thm:genericallyci}
	Let $I\subset S$ be a graded perfect ideal of height two with $\mu(I)=m$. Let $X$ be an Hilbert-Burch matrix of $I$. Set $R=S/I$, and assume that $R$ is generically Gorenstein. Then 
	\[
	\tr(\omega_R)=I_{m-2}(X)/I.
	\]
\end{Theorem}

This result was already shown in \cite[Corollary 3.5]{FHSV22} with a different method. Here, we provide a different and somewhat easier argument.\medskip

We shall need the following
\begin{Lemma}\label{Lem:genGonCI}
	Let $I\subset S$ be a Cohen-Macaulay ideal of height two. Then, the following conditions are equivalent.
	\begin{enumerate}
		\item[\textup{(a)}] $I$ is generically Gorenstein.
		\item[\textup{(b)}] $I$ is generically a complete intersection.
	\end{enumerate}
\end{Lemma}
\begin{proof}
	That (b) implies (a) is clear. Conversely, assume (a) and let $P\in\Ass(I)$ be a minimal associated prime. Then $I_P$ is again a Cohen-Macaulay ideal of height two, and by assumption $S_P/I_P$ is Gorenstein. Thus $\mu(I_P)\ge2$ and the Cohen-Macaulay type of $I_P$ is one. By the Hilbert-Burch Theorem \ref{Thm:HB} we deduce that $I_P$ is generated by two elements. Hence $I_P$ is a complete intersection.
\end{proof}

Let $I\subset S$ be a perfect graded ideal of height two, generically a complete intersection  (equivalently, generically Gorenstein), and minimally generated by $m$ homogeneous polynomials $\fb=f_1,\ldots,f_m$ with $\deg f_i=a_i$ for $1\le i\le m$. Furthermore, let $X$ be an Hilbert-Burch matrix of $I$ with homogeneous entries. Set $R=S/I$, and consider the graded minimal free $S$-resolution 
\[
0\rightarrow\bigoplus_{j=1}^{m-1}S(-b_j)\xrightarrow{X}\bigoplus_{i=1}^mS(-a_i)\rightarrow I\rightarrow0
\]
of $I$, and the induced exact sequence 
\[
0\rightarrow C\rightarrow\bigoplus_{j=1}^{m-1}R(-b_j)\xrightarrow{\psi}\bigoplus_{i=1}^mR(-a_i)\rightarrow I/I^2\rightarrow0,
\]
where $C=\Ker\psi$ and $I/I^2$ is the so-called \textit{conormal module} of $I$.

\begin{proof}[Proof of Theorem \ref{Thm:genericallyci}]
Let $X'$ be the $(m-1)\times m$-matrix with entries the indeterminates $x_{ij}$, for $1\le i\le m$ and $1\le j\le m-1$, and let $S'=S[X']$ and $R'=S'/I'$, where $I'=I_{m-1}(X')$. We set $\deg x_{ij}=b_j-a_i$ for all $i$ and $j$. Then $I'$ is a  graded ideal in the non-standard graded polynomial ring $S'$ and $R$ is a specialization of $R'$. In other words, the sequence ${\bf g}=x_{11}-f_{11},\dots,x_{m,m-1}-f_{m,m-1}$ is a homogeneous regular sequence on $S'$ and $R'$ with $S'/({\bf g})=S$ and $R'/({\bf g})=R$.

We have the exact sequence
\[
0\rightarrow\bigoplus_{j=1}^{m-1}S'(-b_j)\xrightarrow{X'}\bigoplus_{i=1}^mS'(-a_i)\rightarrow I'\rightarrow0
\]
and the induced exact sequence
\begin{eqnarray}\label{sequenceprime}
0\rightarrow C'\rightarrow\bigoplus_{j=1}^{m-1}R'(-b_j)\xrightarrow{\psi'}\bigoplus_{i=1}^mR'(-a_i)\rightarrow I'/I'^2\rightarrow0,
\end{eqnarray}
where $C'=\Ker\psi'$.

We denote by the overline ``$\overline{\phantom{ll}}$" the reduction modulo $({\bf g})$. Then $\overline{R'}=R$, and the exact sequence (\ref{sequenceprime}) induces the exact sequences
\begin{eqnarray}\label{sequenceone}
\overline{C'}\rightarrow\bigoplus_{j=1}^{m-1}R(-b_j)\rightarrow\overline{\Im\psi'}\rightarrow0
\end{eqnarray}
and
\begin{eqnarray}\label{sequencetwo}
\overline{\Im\psi'}\rightarrow\bigoplus_{i=1}^mR(-a_i)\rightarrow\overline{ I'/I'^2}\rightarrow0.
\end{eqnarray}

Now we are going to show that:
\begin{enumerate}
\item[(i)] $\overline{I'/I'^2}=I/I^2$;
\item[(ii)] $\overline{C'}\rightarrow\bigoplus_{j=1}^{m-1}R(-b_j)$ is injective;
\item[(iii)] $\overline{\Im\psi'}\rightarrow\bigoplus_{i=1}^mR(-a_i)$ is injective.
\end{enumerate}
From (i), (ii) and (iii) it follows that
\[
0\rightarrow\overline{C'}\rightarrow\bigoplus_{j=1}^{m-1}R(-b_j)\xrightarrow{\psi}\bigoplus_{i=1}^mR(-a_i)\rightarrow I/I^2\rightarrow0,
\]
is exact, and this implies that $C=\overline{C'}$. This then yields the desired conclusion.

{\it Proof of} (i): We have the exact sequence
\begin{eqnarray}\label{true}
\overline{I'^2}\rightarrow\overline{I'}\rightarrow\overline{I'/I'^2}\rightarrow0.
\end{eqnarray}
Since ${\bf g}$ is a regular sequence on $R'$ and $S'$, the exact sequence $0\rightarrow I'\rightarrow S'\rightarrow R'\rightarrow0$
induces the exact sequence
\[
0\rightarrow\overline{I'}\rightarrow S\rightarrow R\rightarrow0.
\]
Since the image of the map $\overline{I'}\rightarrow S$ is equal to $I$ we see that $\overline{I'}$ is identified with $I$ via the injective map $\overline{I'}\rightarrow S$. Under this identification, the image of $\overline{I'^2}\rightarrow\overline{I'}$ in the exact sequence (\ref{true}) is identified with $I^2$. Thus, (\ref{true}) shows that $\overline{I'/I'^2}=I/I^2$.

\medskip
{\it Proof of} (ii): Here and in the proof of (iii) we use the following facts: Let $T$ be a polynomial ring over a field $K$, and let $J\subset T$ be a graded perfect ideal of height two, minimally generated by the homogeneous elements ${\bf h}=h_1,\ldots, h_m$. Let $H_i(J;T)$ be the Koszul homology $H_i({\bf h};T)$. This notation is justified because Koszul homologies do not depend on the minimal set of generators of $J$ \cite[Proposition 1.6.21]{BH}. Then,
\begin{enumerate}
\item[(a)] \cite[Theorem (2.1)]{AH} The Koszul homology modules $H_i(J;T)$ are maximal Cohen-Macaulay $T/J$-modules.\smallskip
\item[(b)] \cite[Theorem (3.2)]{AH} Let $\deg h_i=a_i$ for $1\le i\le m$. Then the natural sequence 
\[
0\rightarrow H_1(J;T)\rightarrow\bigoplus_{i=1}^m(T/J)(-a_i)\rightarrow J/J^2\rightarrow0
\]
is exact, if $J$ is generically a complete intersection.
\end{enumerate}

We may apply this to $I'\subset S'$. It follows from (\ref{sequenceprime}) that $\Im\psi'=H_1(I';S')$, since $R'=S'/I'$ is a domain by \cite[Theorem 7.3.1(c)]{BH}. Therefore,
\[
0\rightarrow C'\rightarrow\bigoplus_{j=1}^{m-1}R'(-b_j)\rightarrow\Im\psi'\rightarrow0
\]
is an exact sequence of maximal Cohen-Macaulay modules. Hence this sequence remains exact after reduction modulo the regular sequence ${\bf g}$. This proves (ii).

{\it Proof of} (iii): The statement (iii) is equivalent to saying that the natural map $\overline{\Im \psi'}\rightarrow\Im\psi$ is an isomorphism. Therefore by (b) it amounts to show the the natural map $\overline{H_1(I';S')}\rightarrow H_1(I;S)$ is an isomorphism.

The desired isomorphism follows by induction on the length of the sequence ${\bf g}$ from the following claim: let $J\subset T$ be as above a perfect ideal of height two, which is not necessarily a generically complete intersection. Let $g\in T$ be a homogeneous polynomial of degree $d$ which is regular on $T/J$. Then the natural map
$$
\overline{H_1(J;T)}\rightarrow H_1(J;\overline{T})
$$
is an isomorphism. (Note that $\overline{T/J}=\overline{T}/\overline{J}$ and $H_1(J;\overline{T})=H_1(\overline{J};\overline{T})$).

Observing that for any partial sequence of ${\bf g}$, the residue class ring modulo this partial sequence is again a polynomial ring, we can proceed with the induction by using the statement of the claim.

For the proof of the claim we consider the short exact sequence
\[
0\rightarrow T(-d)\xrightarrow{g}T\rightarrow\overline{T}\rightarrow0,
\]
which induces the long exact sequence of Koszul homology
\begin{eqnarray*}
\cdots\rightarrow H_1(J;T(-d))\xrightarrow{g}H_1(J;T)\rightarrow H_1(J;\overline{T})\rightarrow
H_0(J;T(-d))\xrightarrow{g}H_0(J;T).
\end{eqnarray*}
Note that $H_0(J;T)\cong T/J$. Since we assume that $g$ is regular on $T/J$, the map $H_0(J;T(-d))\xrightarrow{g}H_0(J;T)$ is injective. This implies that
$$
H_1(J;T(-d))\xrightarrow{g}H_1(J;T)\rightarrow H_1(J;\overline{T})\rightarrow0
$$
is exact. Since the kernel of the second map is the submodule of $H_1(J;T)$ generated by $g$, we see that the natural map $\overline{H_1(J;T)}\rightarrow H_1(J;\overline{T})$ is an isomorphism.
\end{proof}

On the other hand, the inclusion $I_{m-2}(X)R\subseteq\tr(\omega_R)$ is always true, without the generically Gorenstein assumption on $R$. For $m\ge1$, we set $[m]=\{1,2,\dots,m\}$.
\begin{Lemma}\label{Lem:c_A}
	Let $I\subset S$ be a graded perfect ideal of height two with $\mu(I)=m$ and let $a\in S$ be a non-zero element. Set $J=aI$ and $R=S/J$. The minimal free resolution of $J$ induces an exact sequence
	$$
	0\rightarrow C\xrightarrow{\alpha}R^{m-1}\xrightarrow{\psi}R^m\rightarrow J/J^2\rightarrow 0,
	$$
	where $C=\Ker\psi$ and $\alpha$ is the inclusion map. Then
	\[
	aI_{m-2}(\psi)\subseteq I_1(\alpha).
	\]
\end{Lemma}
\begin{proof}
	Let ${\bf f}_1,\dots,{\bf f}_{m-1}$ and ${\bf e}_1,\dots,{\bf e}_m$ be two bases for $R^{m-1}$ and $R^{m}$. By the Hilbert-Burch Theorem \ref{Thm:HB}, $\psi$ is represented by the image in $R$ of an homogeneous $m\times(m-1)$ Hilbert-Burch matrix $X=(a_{ij})$ of $I$ with respect to the given bases. For an element $f\in S$, we denote again by $f$ the image in $R$.
	
	For two subsets $A\subset[m]$, $B\subset[m-1]$ of the same size, we denote by $[A|B]$ the minor of $X$ whose rows are indexed by $A$ and whose columns are indexed by $B$.
	
	For any $A\subset[m]$ with $|A|=m-2$ we set
	\[
	c_A\ =\ \sum_{j=1}^{m-1}(-1)^{j+1}a[A|[m-1]\setminus\{j\}]{\bf f}_j.
	\]
	Note that 
	\begin{eqnarray*}
		\psi(c_A) &=& \sum_{j=1}^{m-1}(-1)^{j+1}a[A|[m-1]\setminus\{j\}]
		\sum_{i=1}^ma_{ij}{\bf e}_i\\
		&=& \sum_{i=1}^m(\sum_{j=1}^{m-1}
		a(-1)^{j+1}a_{ij}[A|[m-1]\setminus\{j\}]){\bf e}_i\\
		&=& \sum_{i=1}^m\pm a[A_i|[m-1]]{\bf e}_i,
	\end{eqnarray*}
	where $A_i$ is obtained from $A$ by adding the $i$th column of $X$ to $A$. If $i\in A$, then $[A_i|[m-1]]=0$, and if $i\notin A$, then $a[A_i|[m-1]]\in aI_{m-1}(X)=aI=J$ and hence is equal to $0$ in $R$. Therefore, $c_A\in C$ for all $A\subset [m]$ with $|A|=m-2$. This shows that $aI_{m-2}(\psi)\subseteq I_1(\alpha)$, and concludes the proof.
\end{proof}

Next, we provide a structure theorem for perfect monomial ideals of height two which are generically Gorenstein. To state this result, we need some preparation.

Hereafter, $S$ denotes the standard graded polynomial ring. That is $\deg x_i=1$ for all $i$. Let $u={\bf x^a}=x^{a_1}\cdots x_n^{a_n}\in S$ be a monomial, where ${\bf a}=(a_1,\dots,a_n)\in\ZZ_{\ge0}^n$. The \textit{polarization} of $u$ is the monomial
$$
u^\wp=\prod_{i=1}^n(\prod_{j=1}^{a_i}x_{i,j})=\prod_{\substack{1\le i\le n\\ a_i>0}}x_{i,1}x_{i,2}\cdots x_{i,a_i}.
$$

Whereas, the \textit{polarization} of a monomial ideal $I\subset S$, with minimal monomial generating set $G(I)$, is defined as the monomial ideal $I^\wp$ with minimal generating set $G(I^\wp)=\{u^\wp:u\in G(I)\}$, in the polynomial ring $S^\wp$ in the variables $x_{i,j}$. Given monomial ideals $I_1,\dots,I_t\subset S$, we have $(I_1\cap\dots\cap I_t)^\wp=I_1^\wp\cap\dots\cap I_t^\wp$.

Let $G$ be a finite simple graph with vertex set $\{x_1,\dots,x_n\}$ with $t$ edges, and let ${\bf a}:a_1,\dots,a_t$ and ${\bf b}:b_1,\dots,b_t$ be two sequences of positive integers. Suppose that $E(G)=\{\{x_{i_\ell},x_{j_\ell}\}:1\le\ell\le t\}$. We associate to $G$, ${\bf a}$ and ${\bf b}$ the height-unmixed height two monomial ideal of $S=K[x_1,\dots,x_n]$,
\begin{equation}\label{eq:G(a,b)-I}
	I_{G,{\bf a},{\bf b}}\ =\ \bigcap_{\ell=1}^t(x_{i_\ell}^{a_\ell},x_{j_\ell}^{b_\ell}).
\end{equation}
We define the \textit{intersection graph} $G({\bf a,b})$ of $I_{G,{\bf a},{\bf b}}$ as follows:

The vertex set of $G({\bf a,b})$ is
$$
V(G({\bf a,b}))\ =\ \bigcup_{i=1}^n\big\{x_{i,j}\ :\ 1\le j\le\max_{\ell,k}(\{a_\ell:a_\ell=i\}\cup\{b_k:b_k=i\})\big\},
$$
and the edge set of $G({\bf a,b})$ is
$$
E(G({\bf a,b}))\ =\ \bigcup_{\ell=1}^t(\bigcup_{p=1}^{a_\ell}\bigcup_{q=1}^{b_\ell}\{x_{i_\ell,p},x_{j_\ell,q}\}).
$$

Finally, we recall that a graph $G$ is called \textit{chordal} if it has no induced cycles of length bigger than three. Whereas, $G$ is called \textit{cochordal} if the complementary graph $G^c$ is chordal. Here, the \textit{complementary graph} $G^c$ of $G$ is the graph with the same vertex set of $G$ and whose edges are the non edges of $G$.

For unexplained facts of the theory of Alexander duality on monomial ideals we refer the reader to \cite[Subsections 1.5.3 and 8.1.2]{HH2011}.
\begin{Theorem}\label{Thm:G(a,b)}
	Let $I\subset S=K[x_1,\dots,x_n]$ be a monomial ideal. Then, the following conditions are equivalent.
	\begin{enumerate}
		\item[\textup{(a)}] $I$ is a perfect monomial ideal of height two which is generically a complete intersection.
		\item[\textup{(b)}] $I$ is a perfect monomial ideal of height two which is generically Gorenstein.
		\item[\textup{(c)}] $I=I_{G,{\bf a,b}}$ for some graph $G$ with $t$ edges, and some integer sequences ${\bf a}$, ${\bf b}$ of length $t$ such that $G({\bf a,b})$ is a cochordal graph.
	\end{enumerate}
\end{Theorem}
\begin{proof}
	(a) $\Rightarrow$ (b) This implication is trivial.
	
	(b) $\Rightarrow$ (c) Since $I$ is perfect, $S/I$ is Cohen-Macaulay, and thus $I$ is height-unmixed of height two. We are going to prove that $I=I_{\bf G,{\bf a,b}}$ for some graph $G$ and some integer sequences ${\bf a,b}$ such that $G({\bf a,b})$ is cochordal. Let $I=Q_1\cap\dots\cap Q_t$ be the standard primary decomposition of $I$, cfr \cite[Subsection 1.3.2, page 12]{HH2011}. Then $Q_\ell=(x_{i_\ell}^{a_\ell},x_{j_\ell}^{b_\ell})$ for all $1\le\ell\le t$. Let $P_\ell=\sqrt{Q_\ell}=(x_{i_\ell},x_{j_\ell})$ for all $\ell$. Hence $\textup{Ass}(I)=\{P_1,\dots,P_t\}$. Applying localization and \cite[Proposition 4.9]{AtiyahMacDonald} we obtain
	\begin{equation}\label{eq:PrimDP_i}
		I_{P_k}=(\bigcap_{\ell=1}^tQ_\ell)_{P_k}=\bigcap_{\ell=1}^{t}(Q_\ell)_{P_k}=\bigcap_{\ell\ :\ P_\ell=P_k}(Q_\ell)_{P_k}.
	\end{equation}
	
	We may replace ordinary localization with the so-called monomial localization. Indeed, recall that for a monomial ideal $L\subset S$, the \textit{monomial localization} of $L$ with respect to $P=(x_{i_1},\dots,x_{i_r})$ is the monomial ideal $L(P)$ in the polynomial ring $S(P)=K[x_{i_1},\dots,x_{i_r}]$ obtained by applying the substitutions $x_i\mapsto 1$ for $x_i\notin P$. Since we have the equality $L(P)S_P=L S_P$, we can indeed replace ordinary localization with monomial localization. Thus, equation (\ref{eq:PrimDP_i}) becomes
	\begin{equation}\label{eq:I(P_k)}
		I(P_k)=\bigcap_{\ell\ :\ P_\ell=P_k}Q_\ell.
	\end{equation}

	Since $I$ is generically Gorenstein, $I(P_k)$ must be a Gorenstein ideal of $K[x_{i_k},x_{j_k}]$. Let $p$ be the projective dimension of $I(P_k)$. The Cohen-Macaulay type of $I(P_k)$ must be $\beta_p(I(P_k))=1$. Notice that $I(P_k)$ has at least two generators. It follows from Lemma \ref{Lem:X-HilBur}, in the next section, that $p=1$ and $\mu(I(P_k))=\beta_1(I(P_k))+1$ is exactly two. To simplify the notation, put $x=x_{i_\ell}$ and $y=x_{j_\ell}$. Thus $I(P_k)=(x^{c_1}y^{d_1},x^{c_2}y^{d_2})$ with $c_1>c_2\ge0$ and $0\le d_1<d_2$. Since $I(P_k)$ is Cohen-Macaulay and not a principal ideal, Lemma \ref{Lem:S/I_{a,b}CM} implies that $c_2=d_1=0$. Thus $I(P_k)=(x^{c_1},y^{d_2})$ is $(x,y)$-primary and since (\ref{eq:I(P_k)}) is again a standard primary decomposition of $I(P_k)$, it follows that there is only one $\ell$ such that $P_\ell=P_k$, namely $\ell=k$.
	
	Summarizing our argument, we have shown that all components $Q_1,\dots,Q_t$ have a different radical. Thus $\{i_p,j_p\}\ne\{i_q,j_q\}$ for all $1\le p<q\le t$. This fact shows that $I=I_{G,{\bf a,b}}$ for some finite simple graph $G$ on vertex set $\{x_1,\dots,x_n\}$ with $t$ edges and with ${\bf a}:a_1,\dots,a_t$ and ${\bf b}:b_1,\dots,b_t$.
	
	Since polarization commutes with intersections, we have
	$$
	I^\wp=I_{G,{\bf a,b}}^\wp=\bigcap_{\ell=1}^t(\prod_{p=1}^{a_\ell}x_{i_\ell,p},\prod_{q=1}^{b_\ell}x_{j_\ell,q})=\bigcap_{\ell=1}^t[\bigcap_{p=1}^{a_\ell}\bigcap_{q=1}^{b_\ell}(x_{i_\ell,p},x_{j_\ell,q})]=I_{G({\bf a,b}),{\bf 1,1}},
	$$
	where ${\bf 1}:1,1,\dots,1$ is a sequence of length $\sum_{\ell=1}^ta_\ell b_\ell$.
	
	It follows from \cite[Corollary 1.6.3]{HH2011} that $S/I$ is a Cohen-Macaulay ring if and only if $S^\wp/I^\wp$ is such. Notice that the Alexander dual $(I^\wp)^\vee$ of $I^\wp$ is the edge ideal of the graph $G({\bf a,b})$. It follows from the Eagon-Reiner criterion \cite[Theorem 8.1.9]{HH2011} that $S^\wp/I^\wp$ is Cohen-Macaulay if and only if $(I^\wp)^\vee$ has a linear resolution. By Fr\"oberg theorem \cite[Theorem 9.2.3]{HH2011} this is equivalent to $G({\bf a,b})$ being a cochordal graph. Assertion (c) follows.
	
	(c) $\Rightarrow$ (a) Suppose $I=I_{G,{\bf a,b}}$ and $G({\bf a,b})$ is cochordal. It follows by the very definition and structure of $I_{G,{\bf a,b}}$, that this ideal has height two and that $S/I$ is generically a complete intersection. As before, we have $I^\wp=I_{G({\bf a,b}),{\bf 1,1}}$ and by Fr\"oberg theorem \cite[Theorem 9.2.3]{HH2011} $(I^\wp)^\vee$ has linear resolution because it is the edge ideal of a cochordal graph. By \cite[Corollary 1.6.3]{HH2011} and the Eagon-Reiner criterion \cite[Theorem 8.1.9]{HH2011} we conclude that $I$ is perfect.
\end{proof}

\section{The canonical trace of Cohen-Macaulay algebras $K[x,y]/I$}\label{sec3}

Let $S=K[x,y]$. The set of monomial ideals $I\subset S$ is in bijection with the set of all pairs $({\bf a},{\bf b})$ of integer sequences satisfying the conditions
\begin{equation}\label{eq:a-b}
	{\bf a}:a_1>a_2>\cdots > a_m\ge 0\,\,\,\,\,\,\,\,\textup{and}\,\,\,\,\,\,\,\,{\bf b}:0\le b_1<b_2<\cdots<b_m.
\end{equation}
Indeed, if $I$ is a monomial ideal and $G(I)$ is its minimal monomial generating set, then $G(I)=\{x^{a_1}y^{b_1},x^{a_2}y^{b_2},\dots,x^{a_m}y^{b_m}\}$ for two sequences ${\bf a}$ and ${\bf b}$ as above. Conversely, if ${\bf a}$ and ${\bf b}$ are two sequences as above, then $x^{a_1}y^{b_1},x^{a_2}y^{b_2},\dots,x^{a_m}y^{b_m}$ is the minimal generating set of a monomial ideal $I\subset S$.

Hereafter, we set $I_{\bf a,b}=(x^{a_1}y^{b_1},x^{a_2}y^{b_2},\dots,x^{a_m}y^{b_m})$, where ${\bf a}$ and ${\bf b}$ are as in (\ref{eq:a-b}).\medskip

The aim of this section is to prove:
\begin{Theorem}
	\label{Thm:mainzero}
	Let $I\subset S=K[x,y]$ be a monomial ideal such that $R=S/I$ is Cohen-Macaulay. Let $X$ be an Hilbert-Burch matrix of $I$. Then
	$$
	\tr(\omega_R)=I_{\mu(I)-2}(X)/I.
	$$
\end{Theorem}

Theorem \ref{Thm:mainzero} follows from the following slightly more general statement.
\begin{Theorem}\label{Thm:mainzero1}
	Let $I=I_{\bf a,b}\subset S=K[x,y]$ be a monomial ideal with $\mu(I)=m\ge2$. Let $\FF:0\rightarrow S^{m-1}\xrightarrow{\varphi}S^m\rightarrow I\rightarrow 0$ be the minimal free resolution of $I$. Set $R=S/I$. We have the exact sequence
	$$
	0\rightarrow C\xrightarrow{\alpha}R^{m-1}\xrightarrow{\psi}R^m\rightarrow I/I^2\rightarrow0,
	$$
	where $C=\Ker\psi$, $\alpha$ is the inclusion map and  $\psi=\varphi\otimes\textup{id}_R$. In particular,
	$$
	I_1(\alpha)\ =\ x^{a_m}y^{b_1}I_{m-2}(\psi).
	$$
\end{Theorem}
For the proof of this result we need the following lemma.

\begin{Lemma}\label{Lem:X-HilBur}
	Let $m\ge2$. The minimal free resolution $\FF$ of $I_{\bf a,b}$ is
	$$
	\FF\ \ :\ \ 0\rightarrow\bigoplus_{i=1}^{m-1}S{\bf f}_i\xrightarrow{\varphi}\bigoplus_{i=1}^{m}S{\bf e}_i\xrightarrow{\varepsilon}I_{\bf a,b}\rightarrow 0,
	$$
	where ${\bf f}_i$ has multidegree $x^{a_i}y^{b_{i+1}}$, for $1\le i\le m-1$, ${\bf e}_i$ has multidegree $x^{a_i}y^{b_i}$ and $\varepsilon({\bf e}_i)=(-1)^ix^{a_i}y^{b_i}$, for $1\le i\le m$, and $\varphi$ is represented with respect to the given bases by the following Hilbert-Burch matrix:
	
	\begin{equation}\label{eq:HBmatrixK[x,y]}
	X\ =\ \left(\begin{matrix}
		\phantom{-}y^{b_2-b_1} &    &       &   \\
		-x^{a_1-a_2} & \phantom{-}y^{b_3-b_2}&        &  \\
		& -x^{a_2-a_3} & \ddots    & \\
		&      & \ddots  & \phantom{-}y^{b_{m}-b_{m-1}} \\
		&    & &-x^{a_{m-1}-a_m}
	\end{matrix}\right).
	\end{equation}
\end{Lemma}

\begin{proof}
Let ${\bf c}:c_1>\dots>c_{m-1}>c_m=0$ and ${\bf d}:0=d_1<d_2<\dots<d_m$, where $c_i=a_i-a_m$ and $d_i=b_i-b_1$ for $1\le i\le m$. Since $x^{a_1-a_m},y^{b_m-b_1}\in I_{\bf c,d}$ and these elements form a regular sequence, we have $\depth I_{\bf c,d}=2$.

It is easily seen that $x^{a_{m}}y^{b_1}\cdot[[m]\setminus\{i\}\,|\,[m-1]]=(-1)^{m-i}x^{a_i}y^{b_i}$ for all $1\le i\le m$. Therefore $x^{a_m}y^{b_1}I_{m-2}(X)=I_{\bf a,b}$. Since $I_{m-2}(X)=I_{\bf c,d}$ and $\depth I_{m-2}(X)=2$, Theorem \ref{Thm:HB} implies that $\FF$ is indeed the minimal free resolution of $I_{\bf a,b}$.
\end{proof}

Given integers $k\le h$, we denote by $[k,h]$ the interval $\{k,k+1,\dots,h\}$.

Let $\psi=\varphi\otimes\textup{id}_R$. We denote the bases $\{{\bf f}_i\otimes 1\}_{1\le i\le m-1}$ and $\{{\bf e}_i\otimes 1\}_{1\le i\le m}$ again by $\{{\bf f}_i\}_{1\le i\le m-1}$ and $\{{\bf e}_i\}_{1\le i\le m}$, respectively. A matrix representing $\psi$ with respect to these bases is the $m\times(m-1)$ matrix whose entries are the images in $R=S/I$ of the entries of the matrix $X$ given in (\ref{eq:HBmatrixK[x,y]}). We denote this matrix again by $X$, and we denote the residue class of any element $f\in S$ in $R$ again by $f$.

\begin{proof}[Proof of Theorem \ref{Thm:mainzero1}]
	Let ${\bf w}=\sum_{i=1}^{m-1}w_i{\bf f}_i\in\Ker\psi$. Proceeding by induction on $m\ge2$, we will show that $w_i\in x^{a_m}y^{b_1}I_{m-2}(\psi)$ for all $i$, with the base case $m=2$ being trivial. This implies that $I_1(\alpha)\subseteq x^{a_m}y^{b_1}I_{m-2}(\psi)$. Since the opposite inclusion holds by Lemma \ref{Lem:c_A}, the desired equality follows.
	
	By Lemma \ref{Lem:X-HilBur}, $I_j(\psi)=I_j(X)$ for all $j$, where $X$ is given in (\ref{eq:HBmatrixK[x,y]}).
	
	The resolution $\FF$ of $I$ is multigraded. Therefore, we may assume that each $w_i$ is a monomial of $R$. Let $p=\min\{i:w_i\ne0\}$ and $q=\max\{i:w_i\ne0\}$. Then\smallskip
	\begin{align*}
		\psi({\bf w})\ &=\ (y^{b_{p+1}-b_p}w_p)\,{\bf e}_p\\
		&+\ \sum_{i=p+1}^{q}(y^{b_{i+1}-b_i}w_i-x^{a_{i-1}-a_i}w_{i-1})\,{\bf e}_i\\[0.25em]
		&-\ (x^{a_{q}-a_{q+1}}w_{q})\,{\bf e}_{q+1}\ =\ 0.
	\end{align*}

	All the coefficients in this expression must be zero in $R=S/I$. Since $I$ is a monomial ideal, by \cite[Corollary 1.1.3]{HH2011} we have
	\begin{equation}\label{eq:w1}
		y^{b_{p+1}-b_p}w_p,\,\,x^{a_{q}-a_{q+1}}w_{q}\in I
	\end{equation}
	and
	$$
	\textup{either}\,\,\,\,\,\,\,\,y^{b_{i+1}-b_i}w_i,\,\,x^{a_{i-1}-a_i}w_{i-1}\in I\,\,\,\,\,\,\,\,\textup{or}\,\,\,\,\,\,\,\,y^{b_{i+1}-b_i}w_i=x^{a_{i-1}-a_i}w_{i-1},
	$$
	for $p<i\le q$.
	
	Notice that if $y^{b_{i+1}-b_i}w_i$, $x^{a_{i-1}-a_i}w_{i-1}\in I$ for some $p<i\le q$, then both the elements ${\bf w}'=\sum_{j=p}^{i-1}w_j{\bf f}_j$ and ${\bf w}''=\sum_{j=i}^{q}w_{j}{\bf f}_j$ are again in $\Ker\psi$. Hence, we may furthermore assume that
	\begin{equation}\label{eq:w2}
		y^{b_{i+1}-b_i}w_i=x^{a_{i-1}-a_i}w_{i-1}\,\,\,\,\,\textup{for all}\,\,\,\,p<i\le q.
	\end{equation}
	
	Let $J=(I:y^{b_{p+1}-b_p})$ and $L=(I:x^{a_{q}-a_{q+1}})$. By \cite[Proposition 1.2.2]{HH2011} we have $J=(x^{a_j}y^{\max\{b_{j}+b_p-b_{p+1},0\}}:1\le j\le m)$ and $L=(x^{\max\{a_{j}-a_q+a_{q+1},0\}}y^{b_j}:1\le j\le m)$. By equation (\ref{eq:w1}), $w_p\in J$ and $w_q\in L$. 
	
	Hence, $x^{a_h}y^{\max\{b_{h}+b_p-b_{p+1},0\}}$ divides $w_p$ for some $1\le h\le m$. Using equation (\ref{eq:w2}) for $i=p+1$, we see that $x^{a_h+a_p-a_{p+1}}$ divides $x^{a_p-a_{p+1}}w_p=y^{b_{p+2}-b_{p+1}}w_{p+1}$ and so divides $w_{p+1}$. Using again (\ref{eq:w2}) for $i=p+2$, we see that $x^{a_h+a_p-a_{p+2}}$ divides $w_{p+2}$. Iterating this reasoning, we obtain that
	\begin{equation}\label{eq:wi1}
		x^{a_h+a_{p}-a_{i}}\,\,\,\,\textup{divides}\,\,\,\,w_{i}\,\,\,\,\,\textup{for all}\,\,\,\,p\le i\le q.
	\end{equation}

	Similarly, $x^{\max\{a_{k}-a_q+a_{q+1},0\}}y^{b_k}$ divides $w_q$ for some $1\le k\le m$. Using (\ref{eq:w2}) for $i=q$ we see that $y^{b_{k}+b_{q+1}-b_q}$ divides $y^{b_{q+1}-b_q}w_q=x^{a_{q-1}-a_q}w_{q-1}$. Hence $y^{b_{k}+b_{q+1}-b_q}$ divides $w_{q-1}$. Using again (\ref{eq:w2}) for $i=q-1$, we see that $y^{b_k+b_{q+1}-b_{q-1}}$ divides $w_{q-2}$. Iterating this procedure, we see that
	\begin{equation}\label{eq:wi2}
		y^{b_{k}+b_{q+1}-b_{i+1}}\,\,\,\,\textup{divides}\,\,\,\,w_{i}\,\,\,\,\,\textup{for all}\,\,\,\,p\le i\le q.
	\end{equation}

    We may assume that the following two conditions are satisfied.
	\begin{enumerate}
		\item[(i)] $k<h$, and
		\item[(ii)] $(\min\{p,k\},\max\{q+1,h\})\ne(1,m)$.
	\end{enumerate}
	
	Indeed, suppose (i) does not hold. Then $k\ge h$. Since $p\le i$, then $a_p-a_i\ge0$. Thus, by equation (\ref{eq:wi1}), $x^{a_h}$ divides $w_i$ for $p\le i\le q$. Similarly, since $i\le q$, then $b_{q+1}-b_{i+1}\ge0$, and from equation (\ref{eq:wi2}) we obtain that $y^{b_k}$ divides $w_i$ for $p\le i\le q$. Hence $x^{a_{h}}y^{b_k}$ divides $w_i$ for all $i$. Since $h\le k$, then $a_k\le a_h$ and so $x^{a_k}y^{b_k}$ divides $w_i$ for all $i$. Notice that
	$$
	x^{a_k}y^{b_k}\ =\ x^{a_m}y^{b_1}\cdot x^{a_k-a_m}y^{b_k-b_1}\ =\ \pm\, x^{a_m}y^{b_1}\cdot[[m]\setminus\{k\}\,|\,[m-1]].
	$$
	Hence $w_i\in x^{a_m}y^{b_1}I_{m-1}(X)\subset x^{a_m}y^{b_1}I_{m-2}(X)$ for all $i$.\smallskip
	
	Next, suppose (ii) does not hold. Then, one of the following possibilities occurs: (a) $(p,q)=(1,m-1)$, (b) $(p,h)=(1,m)$, (c) $(k,q)=(1,m-1)$, (d) $(k,h)=(1,m)$.
	
	Assume that (a) holds. Since $a_h\ge a_m$ and $b_k\ge b_1$, the equations (\ref{eq:wi1}) and (\ref{eq:wi2}) imply that $x^{a_1-a_i+a_m}y^{b_{m}-b_{i+1}+b_1}$ divides $w_i$ for all $i$. Since
	$$
	x^{a_m}y^{b_1}\cdot x^{a_1-a_i}y^{b_{m}-b_{i+1}}\ =\ \pm\,x^{a_m}y^{b_1}\cdot[[m]\setminus\{1,m\}\,|\,[m-1]\setminus\{i\}]
	$$
	for all $i$, we conclude that $w_i\in x^{a_m}y^{b_1} I_{m-2}(X)$ for all $i$. The cases (b), (c) and (d) can be treated similarly.
	
	The conditions (i) and (ii) imply that $[\min\{p,k\},\max\{q+1,h\}]$ is a proper subset of $[m]$. Say it is contained in either $[r,m]$ for some $r>1$ or $[1,s]$ for some $s<m$. 
	
	Assume that $[\min\{p,k\},\max\{q+1,h\}]\subseteq[r,m]$. The other case is analogous. Then $p,k\ge r$. Let $I_1=(x^{a_r}y^{b_r},x^{a_{r+1}}y^{b_{r+1}},\dots,x^{a_m}y^{b_m})$ and set $R_1=S/I_1$. Let $\FF_1:0\rightarrow S^{m-r}\xrightarrow{\varphi_1}S^{m-r+1}\rightarrow I_1\rightarrow0$ be the minimal free resolution of $I_1$ described in Lemma \ref{Lem:X-HilBur}, where $S^{m-r}$ has basis ${\bf f}_{1,1},\dots,{\bf f}_{1,m-r}$, $S^{m-r+1}$ has basis ${\bf e}_{1,1},\dots,{\bf e}_{1,m-r+1}$ and $\varphi_1$ is described by the Hilbert-Burch matrix $X_1$ given in (\ref{eq:HBmatrixK[x,y]}). Notice that $X_1$ is the submatrix of $X$ obtained by deleting the first $r-1$ rows and columns. Let $C_1$ be the kernel of $\psi_1=\varphi_1\otimes\textup{id}_{R_1}$. Since $h>k\ge r$, equations (\ref{eq:w1}) and (\ref{eq:w2}) imply that the element ${\bf w}_1=\sum_{j=p}^qw_j{\bf f}_{1,j}\in C_1$. Since $\mu(I_1)=m-r<\mu(I)$, by induction on $m$ we obtain that each $w_i$ is divided by a monomial $v_i=x^{a_{m}}y^{b_{r}}M_i\in x^{a_{m}}y^{b_{r}}I_{m-(r-1)-2}(X_1)$ for all $p\le i\le q$, where $M_i$ is a $(m-(r-1)-2)$-minor of $X_1$. Notice that
	
	\begin{equation}\label{eq:HBmatrixK[x,y]1}
		X\ =\ \left(\begin{array}{ccc|ccccc}
			\phantom{-}y^{b_2-b_1}&&&&&&&\\
			-x^{a_1-a_2} & \ddots&&&&&&\\
			&\ddots&\!\!\!\phantom{-}y^{b_{r}-b_{r-1}}&&&&&\\ \hline
			&&\!\!\!-x^{a_{r-1}-a_r}&&&&&\\ 
			&&&&&&&\\
			&&&&&\,\,{\Huge X_1}&&\\[0.7pt]
			&&&&&&&
		\end{array}\right).
	\end{equation}\smallskip

    \noindent For all $i$, let $N_i$ be the $(m-2)$-minor of $X$ whose rows are the first $r-1$ rows of $X$ together with those of $M_i$, and whose columns are the first $r-1$ columns of $X$ together with those of $M_i$. From equation (\ref{eq:HBmatrixK[x,y]1}), we see that $y^{b_r-b_1}M_i=N_i$ for all $i$. Hence $w_i\in x^{a_m}y^{b_1}I_{m-2}(X)$ for all $i$, as desired.
\end{proof}

\begin{Lemma}\label{Lem:S/I_{a,b}CM}
	The ring $K[x,y]/I_{\bf a,b}$ is Cohen-Macaulay if and only if $m=1$ or $m>1$ and $a_m=b_1=0$.
\end{Lemma}
\begin{proof}
	Set $S=K[x,y]$ and $I=I_{\bf a,b}$. If $m=1$, then $I$ is a principal ideal and thus $S/I$ is Cohen-Macaulay. Let $m>1$. Then \cite[Proposition 5.1(c)]{FS2023} implies that $\mm=(x,y)\in\Ass(I)$. Since $S/I$ is Cohen-Macaulay, $I$ is height-unmixed, and so $(x),(y)\notin\Ass(I)$. Hence \cite[Proposition 5.1(a)-(b)]{FS2023} implies that $a_m=b_1=0$. Conversely, if $a_m=b_1=0$, then $x^{a_1},y^{b_m}\in I$ and $\dim S/I=0$. Thus $\depth S/I=0$ as well, and $S/I$ is Cohen-Macaulay.
\end{proof}
\begin{proof}[Proof of Theorem \ref{Thm:mainzero}]
	Let $I=I_{\bf a,b}$ and $\mu(I)=m$. If $m=1$, then $R$ is Gorenstein and $\tr(\omega_R)=R=I_{-1}(X)/I$. Let $m>1$. Then $a_m=b_1=0$ by Lemma \ref{Lem:S/I_{a,b}CM}, and the assertion follows from Theorem \ref{Thm:mainzero1} and \cite[Corollary 3.2]{HHS1}.
\end{proof}

We have the following classification.
\begin{Proposition}\label{Prop:I_{a,b}-NG}
	Let $I\subset S=K[x,y]$ be a monomial ideal such that $R=S/I$ is Cohen-Macaulay. Then $R$ is nearly Gorenstein if and only if $\mu(I)=1$, $\mu(I)=2$ or $\mu(I)=3$ and $I=(x^{a},x^{b}y^{c},y^{d})$ for some integers $a>b\ge0$, $0\le c<d$ satisfying the conditions: \textup{(i)} $a-b=1$ or $b=1$, \textup{(ii)} $d-c=1$ or $c=1$, and \textup{(iii)} $b+c\ge1$.
\end{Proposition}
\begin{proof}
	If $\mu(I)=1$ or $\mu(I)=2$, then $R$ is a complete intersection, and thus a (nearly) Gorenstein ring. If $\mu(I)>3$, then $\tr(\omega_R)\subseteq\mm^2$, where $\mm=(x,y)R$. Finally assume $\mu(I)=3$. Then, by Lemma \ref{Lem:S/I_{a,b}CM}, $I=(x^{a_1},x^{a_2}y^{b_2},y^{b_3})$ where $a_1>a_2\ge0$, $0\le b_2<b_3$ and $a_2+b_2\ge1$. By equation (\ref{eq:HBmatrixK[x,y]}), an Hilbert-Burch matrix of $I$ is
	$$
	\left(\begin{matrix}
		\phantom{-}y^{b_2}&0\\
		-x^{a_1-a_2}&\phantom{-}y^{b_3-b_2}\\
		0&-x^{a_2}
	\end{matrix}\right).
	$$
	Hence $\tr(\omega_R)=(x^{a_1-a_2},x^{a_2},y^{b_3-b_2},y^{b_2})/I$, and the assertion follows.
\end{proof}
\section{Nearly Gorenstein monomial ideals of height two}\label{sec4}

We conclude the paper with the classification of the nearly Gorenstein monomial ideals of height two. For a monomial $u\in S$, we set $\supp(u)=\{x_i:x_i\ \textup{divides}\ u\}$.
\begin{Theorem}\label{Thm:I(G(a,b))-NG}
	Up to a relabeling of the variables, the only nearly Gorenstein monomial ideals in $S=K[x_1,\dots,x_n]$ of height two are:
	\begin{enumerate}
		\item[\textup{(a)}] $(u,v)$ for two monomials $u,v\in S$ with $\supp(u)\cap\supp(v)=\emptyset$.
		\item[\textup{(b)}] $(x_1^{a},x_1^{b}x_2^{c},x_2^{d})$ with $n=2$ and $a>b\ge0$, $0\le c<d$ satisfying the conditions: \textup{(i)} $a-b=1$ or $b=1$, \textup{(ii)} $d-c=1$ or $c=1$, and \textup{(iii)} $b+c\ge1$.
		\item[\textup{(c)}] $(x_1^ax_2^b,x_1x_3,x_2x_3)$ with $a\ge1$, $b\ge0$, $a+b\ge2$ and $n=3$.
		\item[\textup{(d)}] $(x_1x_2^b,x_2^{b+1},x_1x_3)$ with $b\ge1$ and $n=3$.
		\item[\textup{(e)}] $(x_1x_3,x_1x_4,x_2x_4)$ and $n=4$.
	\end{enumerate}
\end{Theorem}
\begin{proof}
	Let $I\subset S$ be a nearly Gorenstein monomial ideal of height two and set $R=S/I$. By \cite[Proposition 2.3(a)]{HHS1}, $I$ is Gorenstein on the punctured spectrum, and so is generically Gorenstein. Theorem \ref{Thm:genericallyci} implies that $\tr(\omega_{R})=I_{m-2}(X)R$ where $m=\mu(I)$ and $X$ is an Hilbert-Burch matrix of $I$. Since each non-zero entry of $X$ has positive degree, $\tr(\omega_{R})\subseteq\mm^{m-2}$, with $\mm=(x_1,\dots,x_n)R$. Thus, if $I$ is nearly Gorenstein, then $m\le 3$. We cannot have $m=1$, otherwise $I$ has height one.
	
	If $m=2$, then $I=(u,v)$. If $x_i\in\supp(u)\cap\supp(v)$, then $\height(I)=1$ against the assumption. Thus $\supp(u)\cap\supp(v)=\emptyset$ and case (a) follows. Conversely, if (a) holds, then $I=(u,v)$ is a complete intersection, and thus is (nearly) Gorenstein.
	
	Assume now $m=3$, then $\tr(\omega_{R})\subseteq\mm$ and so we must have $\tr(\omega_R)=\mm$. Since $m=3$, then $n\ge2$. If $n=2$, then Proposition \ref{Prop:I_{a,b}-NG} implies that case (b) holds.
	
	Hence, we may assume $n\ge3$. Let $G(I)=\{u_1,u_2,u_3\}$. Then, we have the following matrix arising from the so-called Taylor resolution (see \cite[Section 7.1]{HH2011}):
	$$
	T=\left(\begin{matrix}
		\phantom{-}\frac{\lcm(u_1,u_2)}{u_1}&\phantom{-}\frac{\lcm(u_1,u_3)}{u_1}&0\\
		-\frac{\lcm(u_1,u_2)}{u_2}&0&\phantom{-}\frac{\lcm(u_2,u_3)}{u_2}\\
		0&-\frac{\lcm(u_1,u_3)}{u_3}&-\frac{\lcm(u_2,u_3)}{u_3}
	\end{matrix}\right).
	$$
	
	It follows from \cite[Lemma 9.2.4]{HH2011} that an Hilbert-Burch matrix of $I$ is an appropriate $3\times 2$ submatrix of $T$. Let $X$ be such an Hilbert-Burch matrix. Then $X$ has only four non-zero entries. Since we must have $\tr(\omega_R)=(x_1,\dots,x_n)R$, but we also have $\tr(\omega_R)=I_1(X)R$, we deduce that $n$ is at most four.
	
	Let $n=3$. Up to relabeling and exchange of rows, we have
	\begin{equation}\label{eq:particularX1}
		X=\left(\begin{matrix}
			\phantom{-}x_1&0\\
			-x_2&\phantom{-}w\\
			0&-x_3
		\end{matrix}\right)\,\,\,\,\,\,\textup{or}\,\,\,\,\,\,X=\left(\begin{matrix}
		\phantom{-}x_1&0\\
		-x_2&\phantom{-}x_3\\
		0&-w
	\end{matrix}\right),
	\end{equation}
	where $w$ is a suitable monomial. 
	
	In the first case $I=I_2(X)=(x_1w,x_1x_3,x_2x_3)$. Note that $x_3$ cannot divide $w$, otherwise we would have $m<3$. Thus $w=x_1^{a}x_2^b$ for some $a,b\ge0$ with $a+b\ge1$. Case (c) follows. Conversely, if (c) holds and $I=(x_1^ax_2^b,x_1x_3,x_2x_3)$ for some $a\ge1$, $b\ge0$ and $n=3$, then $I$ is generically Gorenstein, putting $w=x_1^{a-1}x_2^b$, the first matrix $X$ in (\ref{eq:particularX1}) is an Hilbert-Burch matrix of $I$ and so $I$ is nearly Gorenstein.
	
	In the second case $I=I_2(X)=(x_1x_3,x_1w,x_2w)$. Notice that $x_1$ and $x_3$ cannot divide $w$, otherwise $I$ would have height one. Thus $w=x_2^b$ with $b\ge1$ and case (d) follows. Conversely, if (d) holds and $I=(x_1x_2^b,x_2^{b+1},x_1x_3)$ with $b\ge1$ and $n=3$, then $I$ is generically Gorenstein, putting $w=x_2^b$ the second matrix $X$ given in (\ref{eq:particularX1}) is an Hilbert-Burch matrix of $I$ with $I_1(X)=\mm$ and $I$ is nearly Gorenstein.
	
	Finally, assume $n=4$. Then, up to relabeling and exchange of rows,
	\begin{equation}\label{eq:particularX2}
		X=\left(\begin{matrix}
			\phantom{-}x_1&0\\
			-x_2&\phantom{-}x_3\\
			0&-x_4
		\end{matrix}\right).
	\end{equation}
	Thus $I=I_2(X)=(x_1x_3,x_1x_4,x_2x_4)$ and case (e) follows. Conversely, if (e) holds, then the $I=(x_1x_3,x_1x_4,x_2x_4)$ is generically Gorenstein, the matrix $X$ given in (\ref{eq:particularX2}) is an Hilbert-Burch matrix of $I$ and thus $I$ is nearly Gorenstein.
\end{proof}

{}


\begin{thebibliography}{}

\bibitem{Ahmad} S. Ahmad, Cohen–Macaulay Intersections. Arch. Math., 92:228--236, 2009.

\bibitem{AtiyahMacDonald} M. F. Atiyah, I. G. Macdonald. \textit{Introduction to commutative algebra}. Addison Wesley Publishing Co., Reading, Mass.--London--Don Mills, Ont., 1969.

\bibitem{AH} L. Avramov and J. Herzog, The Koszul algebra of a codimension 2 embedding, Math Z. 175 (1980), 249--260.

\bibitem{BH} W. Bruns and  J. Herzog, \textit{Cohen-Macaulay Rings}, revised ed., Cambridge Stud. Adv. Math., 39, Cambridge University Press, Cambridge, 1998

\bibitem{Burch68} L. Burch. On ideals of finite homological dimension in local rings, Proc. Cambridge Philos. Soc., 64 (4): 941--948, 1968.

\bibitem{CF1} M. Crupi, A. Ficarra, Very well--covered graphs by Betti splittings, J. Algebra {\bf 629}(2023) 76--108. https://doi.org/10.1016/j.jalgebra.2023.03.033.

\bibitem{CRT2011} M. Crupi, G. Rinaldo, N. Terai, Cohen--Macaulay edge ideals whose height is half of the number of vertices, Nagoya Math. J. 201 (2011) 116--130.

\bibitem{D} H. Dao,  T. Kobayashi, R. Takahashi, Trace ideals of canonical modules, annihilators of Ext modules, and classes of rings close to being Gorenstein. J. Pure Appl. Algebra, 225(2021), 106655.

\bibitem{Ei} D.~Eisenbud, {\em Commutative Algebra with a view toward  Algebraic Geometry}, Springer, 1995.

\bibitem{EFHM} D. Eisenbud, A. Ficarra, J. Herzog, S. Moradi. Ideals and their Fitting ideals, 2023, preprint \url{https://arxiv.org/abs/2401.00541}

\bibitem{ET} J. Elias,  MS, Takatuji. On Teter rings. Proceedings of the Royal Society of Edinburgh Section A: Mathematics. {\bf 147} (2017), 125--39.

\bibitem{FHSV22} A. Ficarra, J. Herzog, D.I. Stamate, V. Trivedi, The canonical trace of determinantal rings, 2022, preprint \url{https://arxiv.org/abs/2212.00393}

\bibitem{FS2023} A. Ficarra, E. Sgroi, Asymptotic behaviour of the $\textup{v}$-number of homogeneous ideals, 2023, preprint \url{https://arxiv.org/abs/2306.14243}

\bibitem{GHHM} O. Gasanova, J. Herzog, T. Hibi, S. Moradi. Rings of Teter type. Nagoya Mathematical Journal, {\bf248} (2022), 1005--1033. 

\bibitem{GDS} D.~R.~Grayson, M.~E.~Stillman. {\em Macaulay2, a software system for research in algebraic geometry}. Available at \url{http://www.math.uiuc.edu/Macaulay2}.

\bibitem{HKMM} T. Hall, M. Kölbl, K. Matsushita, S. Miyashita, Nearly Gorenstein polytopes, Electron. J. Comb. 30 (4) (2023).

\bibitem{He} J. Herzog, Generators and relations of abelian semigroups and semigroup rings. Manuscripta Math. \textbf{3}(1970), 175--193.

\bibitem{HH2011} J. Herzog, T. Hibi. \emph{Monomial ideals}, Graduate texts in Mathematics {\bf 260}, Springer, 2011.

\bibitem{HHS1} J. Herzog, T. Hibi, D.I. Stamate, The trace of the canonical module, Israel Journal of Mathematics 233 (2019), 133--165.

\bibitem{HHS2} J. Herzog, T. Hibi, D.I. Stamate, Canonical trace ideal and residue for numerical semigroup rings. Semigroup Forum 103(2021), 550--566.

\bibitem{HV} C. Huneke, A. Vraciu, Rings that are almost Gorenstein,  Pacific J. Math. {\bf 225} (2006), 85--102.

\bibitem{M24} S. Miyashita. Levelness versus nearly Gorensteinness of homogeneous domains. J. Pure Appl. Algebra 228(4), 107553 (2024).

\bibitem{M24b} S. Miyashita. Comparing generalized Gorenstein properties in semi-standard graded rings. Journal of Algebra 647 (2024): 823-843.

\end{thebibliography}
\end{document}